\documentclass[a4paper, 10pt]{amsart}
\usepackage{amsmath,amssymb}
\usepackage{amssymb,amsmath,latexsym,amsthm, hyperref, color}%or any usual package
\usepackage[english]{babel}

\newtheorem{thm}{Theorem}[section]
\newtheorem{theorem}[thm]{Theorem}
\newtheorem{lemma}[thm]{Lemma}

\newtheorem{proposition}[thm]{Proposition}

\theoremstyle{definition}

\numberwithin{equation}{section}

\newcommand{\N}{\mathbb N}

 \DeclareMathOperator{\ord}{ord}

 \DeclareMathOperator{\supp}{supp}

\renewcommand{\t}{\, | \,}

\begin{document}

\title[Sets of minimal distances and  characterizations of class groups]{Sets of minimal distances and \\ characterizations of class groups of Krull monoids }

\author{Qinghai Zhong}
\address{Qinghai Zhong\\
University of Graz, NAWI Graz \\
Institute for Mathematics and Scientific Computing \\
Heinrichstra{\ss}e 36\\
8010 Graz, Austria}
\email{qinghai.zhong@uni-graz.at}
\urladdr{http://qinghai-zhong.weebly.com}

\subjclass[2010]{11B30, 11R27, 13A05, 13F05, 20M13}

\keywords{Krull monoids,    class groups,   arithmetical characterizations, sets of lengths, zero-sum sequences, Davenport constant}

\thanks{This work was supported by the Austrian Science Fund FWF, Project Number P28864-N35.}

\begin{abstract}
Let $H$ be a Krull monoid with finite class group $G$ such that every class contains a prime divisor. Then every non-unit $a \in H$ can be written as a finite product of atoms, say $a=u_1 \cdot \ldots \cdot u_k$. The set $\mathsf L (a)$ of all possible factorization lengths $k$ is called the set of lengths of $a$. There is a constant $M \in \N$ such that all sets of lengths are almost arithmetical multiprogressions with bound $M$ and with difference $d \in \Delta^* (H)$, where $\Delta^* (H)$ denotes the set of minimal distances of $H$.
We study the structure of  $\Delta^* (H)$ and establish a characterization when $\Delta^*(H)$ is an interval.

The system $\mathcal L (H) = \{ \mathsf L (a) \mid a \in H \}$ of all sets of lengths depends only on the class group $G$, and a standing conjecture states that conversely the system $\mathcal L (H)$ is characteristic for the class group. We confirm this  conjecture (among others) if the class group is isomorphic to $C_n^r$ with $r,n \in \N$ and $\Delta^*(H)$ is not an interval.
\end{abstract}

\maketitle

\medskip
\section{Introduction and Main Results} \label{1}
\medskip

Let $H$ be a Krull monoid with finite class group $G$ such that every class contains a prime divisor (holomorphy rings in global fields are such Krull monoids and more examples will be given later).
 Then every non-unit of $H$ has a factorization as a finite product of atoms (or irreducible elements), and all these factorizations are unique (i.e., $H$ is factorial) if and only if $G$ is trivial. Otherwise, there are elements having factorizations which differ not only up to associates and up to the order of the factors. These phenomena are described by arithmetical invariants such as sets of lengths and sets of distances. For an overview of recent developments in Factorization Theory we refer to \cite{C-F-G-O16}.

We  recall some basic concepts and then we formulate the main results of the present paper.
For a finite nonempty set $L = \{m_1, \ldots, m_k\}$ of positive integers with $m_1 < \ldots < m_k$, we denote by $\Delta (L) = \{m_i-m_{i-1} \mid i \in [2,k] \}$ the set of distances of $L$.  If a non-unit $a \in H$ has a factorization $a = u_1 \cdot \ldots \cdot u_k$ into atoms $u_1, \ldots, u_k$, then $k$ is called the length of the factorization, and the set $\mathsf L (a)$ of all possible factorization lengths $k$ is called the set of lengths of $a$.
Since $H$ is  Krull, every non-unit has a factorization into atoms and all sets of lengths are finite. Furthermore, all sets of lengths $\mathsf L (a)$ are singletons if and only if $|G| \le 2$. Suppose that $|G| \ge 3$. Then there is an element $a \in H$ with $|\mathsf L (a)|>1$, and since the $n$-fold sumset $\mathsf L (a) + \ldots + \mathsf L (a)$ is contained in $\mathsf L (a^n)$, it follows that  $|\mathsf L (a^n)| > n$ for every $n \in \N$.
Therefore, the system  $\mathcal L(H)=\{\mathsf L(a)\mid a\in H \}$  of all sets of lengths of $H$ consists of infinitely many finite subsets of the integers, and there are arbitrarily large sets of lengths.

The set of distances $\Delta (H)$ is the union of all sets $\Delta (L)$ over all  $L \in \mathcal L (H)$. Since the class group is finite, $\Delta (H)$ is finite, and since every class contains a prime divisor, $\Delta (H)$ is a finite interval with $\min \Delta (H)=1$ (\cite{Ge-Yu12b}; the maximum of $\Delta (H)$ is unknown in general, see \cite{Ge-Gr-Sc11a,Ge-Zh15b}).
The set of minimal distances $\Delta^* (H)$ is a crucial subset of $\Delta (H)$, defined as
\[
\Delta^* (H) = \{ \min \Delta (S) \mid S \subset H \ \text{is a divisor-closed submonoid with} \ \Delta (S) \ne \emptyset \} \,.
\]
It has been studied by Chapman, Geroldinger, Hamidoune, Schmid et al. (see e.g., \cite[Chapter 6.8]{Ge-HK06a}, \cite{Ge-Ha02, Sc05d,Ch-Sc-Sm08b}),  and the original interest in $\Delta^* (H)$ stemmed from its occurrence in the Structure Theorem for Sets of Lengths. For convenience of the reader we formulate the Structure Theorem    and recall that the given description is best possible (\cite[Chapter 4.7]{Ge-HK06a}, \cite{Sc09a}).

\medskip
\noindent
{\bf Theorem A.} {\it Let $H$ be a Krull monoid with finite class group. Then there is  a constant $M \in \N$ such that the set of lengths $\mathsf L (a)$ of any non-unit $a\in H$ is an {\rm AAMP} $($almost arithmetical multiprogression$)$ with difference $d \in \Delta^* (H)$ and bound $M$.}

\medskip
The last couple of years have seen a renewed interest in $\Delta^* (H)$ partly motivated by the Characterization Problem (which will be discussed below). Among others the maximum of $\Delta^* (H)$ has been determined (we have $\max \Delta^*(H)=\max\{\mathsf r(G)-1, \exp(G)-2\}$ by \cite{Ge-Zh16a}), and a better understanding of $\Delta^* (H)$ opened the door to progress in a variety of directions (e.g, \cite{Ge-Sc16b}).

Whereas the set $\Delta (H)$ of all distances is an interval, the structure of $\Delta^* (H)$ is much more involved.
A simple example shows that the interval $[1, \mathsf r (G)-1]$ is contained in $\Delta^* (H)$ (Lemma \ref{3.2}) and thus $\Delta^* (H)$ is an interval if $\mathsf r (G) \ge \exp (G)-1$.
In the present paper we further study the structure of $\Delta^* (H)$, which allows us to establish a characterization when $\Delta^* (H)$ is an interval. Here is our first main result.

% Cyclic groups are in sharp contrast to this. Indeed, if $G$ is cyclic with $|G|>3$, then $\max \big( \Delta^* (H) \setminus \{|G|-2\}\big) = \lfloor \frac{|G|}{2} \rfloor -1$ (\cite{Ge-Ha02}).

\medskip
\begin{theorem} \label{1.1}
Let $H$ be a Krull monoid with finite class group $G$ such that every class contains a prime divisor.
Suppose that $|G|\ge 3$, $\exp (G)=n$, $\mathsf r (G)=r$, and let $k\in \N$ be maximal such that $G$ has a subgroup isomorphic to $C_{n}^k$. Then \begin{align*}&[1,r-1]\cup\{\max\{1,\lfloor\frac{n}{2}\rfloor-1\}\}\cup [\max\{1,n-k-1\},n-2]\\
\subset \  \Delta^*(H) \ \subset \ &[1,\max\{r-1, \lfloor\frac{n}{2}\rfloor-1 \}]\cup [\max\{1,n-k-1\},n-2]\,.
\end{align*} In particular, the following holds{\rm \,:}
\begin{enumerate}
\item  If $r\ge \left\lfloor\frac{n}{2}\right\rfloor-1$, then \[
    \Delta^*(H)=[1,\max\{r-1, \lfloor\frac{n}{2}\rfloor-1 \}]\cup [\max\{1,n-k-1\},n-2] \,.
    \]

\item The following statements are equivalent{\rm \,:}
\begin{enumerate}
\item[(a)] $\Delta^*(H)$ is an interval.

\item[(b)] $\max\{1,n-k-2\}\in \Delta^*(H)$.

\item[(c)] $n-k-2\le \max\{r-1, \lfloor\frac{n}{2}\rfloor-1\}$.

\item[(d)] $r+k\ge n-1$ or \big( $r+k= n-2$ and $G\cong C_{2r+2}^{r}$\big).
\end{enumerate}

\end{enumerate}
\end{theorem}

\medskip
Thus, in particular, if $\mathsf r (G) \ge \left\lfloor\frac{\exp (G)}{2}\right\rfloor-1$, then $\Delta^* (H)$ is completely determined. However,  if $\mathsf r(G)$ is small with respect to $\lfloor\frac{\exp(G)}{2}\rfloor$, then the structure of $\Delta^*(H)$ remains open. The complexity of this case, even for cyclic groups, can be seen from a recent paper by Plagne and Schmid who studed $\Delta^* (H)$ in case of cyclic class groups (\cite{Pl-Sc16a}).

\smallskip
In order to present our second main result, we recall the Characterization Problem for  class groups.
The monoid $\mathcal B (G)$ of zero-sum sequences over $G$ is  a Krull monoid with class group isomorphic to $G$, every class contains a prime divisor, and the systems of sets of lengths of $H$ and that of $\mathcal B (G)$ coincide. Thus $\mathcal L (H) = \mathcal L \big(B (G) \big)$, and it is usual to set $\mathcal L (G) := \mathcal L \big( \mathcal B (G) \big)$. In particular, the system of sets of lengths of $H$ depends only on the class group $G$. The associated inverse question asks whether or not sets of lengths are characteristic for the class group.
More precisely, the Characterization Problem for class groups can be formulated as follows (for surveys and a detailed description of the background of this problem see \cite[Section 7.3]{Ge-HK06a}, \cite[page 42]{Ge-Ru09}, \cite{Sc09b,Ge16c}).

\smallskip
\begin{enumerate}
\item[]
Given two finite abelian groups $G$ and $G'$ with $|G|\ge 3$ such that $\mathcal L(G) = \mathcal L(G')$.
Does it follow that $G \cong G'$?
\end{enumerate}
\smallskip

The system  $\mathcal L (G)$ is studied with methods from Additive Combinatorics. In particular, zero-sum theoretical invariants (such as the Davenport constant or the cross number) and the associated inverse problems play a crucial role (surveys and detailed presentations of such results can be found in \cite{Ge-HK06a, Ge-Ru09, Gr13a}).
Most of these invariants are well-understood only in a very limited number of cases (e.g., for groups of rank two, the precise value of the Davenport constant $\mathsf D (G)$ is known and the associated inverse problem is solved; however, if $n$ is not a prime power and $r \ge 3$, then the precise value of the Davenport constant $\mathsf D (C_n^r)$ is unknown). Thus it is not surprising that most affirmative answers to the Characterization Problem so far have been restricted to those groups where we have a good understanding of the Davenport constant. These groups include elementary $2$-groups, cyclic groups, and groups of rank two (for recent progress we refer to  \cite{Ge-Sc16a}).

The first groups, for which the Characterization Problem was solved whereas the Davenport constant is unknown, are groups of the form $C_n^r$, where $r, n\in \N$ and $r\le \frac{n+2}{6}$ ( \cite{Ge-Zh16b}). Based on Theorem \ref{1.1} we extend these results and give an affirmative answer to the Characterization Problem for all groups $C_n^r$ for which  $\Delta^*(C_n^r)$ is not an interval.

\smallskip
\begin{theorem} \label{1.2}
Let $G$ and $G'$ be  finite abelian groups and let  $k, k' \in \N$ be maximal such that $G$ has a subgroup isomorphic to $C_{\exp(G)}^k$ and $G'$ has a subgroup isomorphic to $C_{\exp (G')}^{k'}$. Suppose    $\mathsf r(G)+k\le \exp(G)-2$, $G\not\cong C_{2\mathsf r(G)+2}^{\mathsf r(G)}$, and that $\mathcal L(G)=\mathcal L(G')$. Then $\exp(G)=\exp(G')$ and $k=k'$. In particular,
\begin{enumerate}
\item If $\mathsf r(G)\ge \left\lfloor\frac{\exp(G)}{2}\right\rfloor+1$, then $\mathsf r(G)=\mathsf r(G')$.
\item If $G\cong C_{\exp(G)}^{\mathsf r(G)}$, then $G\cong G'$.
\end{enumerate}
\end{theorem}

In Section \ref{2} we gather the required background both  on Krull monoids as well as on Additive Combinatorics as needed in the sequel.  In Section \ref{3} we study structural properties of (large) minimal non-half-factorial subsets of finite abelian groups. Finally the proofs of Theorem \ref{1.1} and \ref{1.2}   will be provided in Section \ref{4}.

\medskip
\section{Background on Krull monoids and their sets of minimal distances} \label{2}
\medskip

Our notation and terminology are consistent with \cite{Ge-HK06a, Ge-Ru09, Gr13a}.
% %
 Let $\mathbb N$ denote the set of positive integers and $\N_0=\N\cup\{0\}$. For $a, b \in \mathbb Q$, we denote
by $[a, b ] = \{ x \in \mathbb Z \mid a \le x \le b\}$ the discrete, finite interval between $a$ and $b$. If $A, B \subset \mathbb Z$ are subsets of the integers, then $A+B = \{a+b \mid a \in A, b \in B\}$ denotes their {\it sumset}, and
% %
  $\Delta (A)$ the {\it set of $($successive$)$ distances} of $A$ (that is, $d \in \Delta (A)$ if and only if $d = b-a$ with $a, b \in A$ distinct and $[a, b] \cap A = \{a, b\}$).

 By a {\it monoid}, we mean a commutative semigroup with identity that satisfies the cancellation laws.
  If $H$ is a monoid, then $H^{\times}$ denotes the unit group and $\mathcal A (H)$  the set of atoms (or irreducible elements) of $H$.
   A submonoid $S \subset H$ is called {\it divisor-closed} if $a \in S$, $b \in H$, and $b$ divides $a$ imply that $b \in S$.
A monoid $H$ is said to be
\begin{itemize}
\item {\it atomic} if every non-unit can be written as a finite product of atoms.

\item {\it factorial} if it is atomic and every atom is prime.

\item {\it half-factorial} if it is atomic and $|\mathsf L (a)|=1$ for each non-unit $a \in H$ (equivalently, $\Delta (H)=\emptyset$).

\end{itemize}
A monoid $F$ is factorial with $F^{\times} = \{1\}$ if and only if it is free abelian. If this holds, then the set of primes $P \subset F$ is a basis of $F$, we write $F = \mathcal F (P)$, and every $a \in F$ has a representation of the form
\[
a = \prod_{p \in P} p^{\mathsf v_p (a)} \quad \text{with} \ \mathsf v_p (a) \in \N_0 \quad \text{and} \quad \mathsf v_p (a) = 0 \ \text{for almost all} \ p \in P \,.
\]

A monoid homomorphism \ $\theta \colon H \to B$ is called a \ {\it
transfer homomorphism} \ if it has the following properties:
\begin{enumerate}

\item[{\bf (T\,1)\,}] $B = \theta(H) B^\times$ \ and \ $\theta
^{-1} (B^\times) = H^\times$.

\smallskip

\item[{\bf (T\,2)\,}] If $u \in H$, \ $b,\,c \in B$ \ and \ $\theta
(u) = bc$, then there exist \ $v,\,w \in H$ \ such that \ $u = vw$, \
$\theta (v) \simeq b$ \ and \ $\theta (w) \simeq c$.

\end{enumerate}
If $H$ and $B$ are atomic monoids and $\theta \colon H \to B$ is a transfer homomorphism, then  (see \cite[Chapter 3.2]{Ge-HK06a})
\[
\mathcal L (H) = \mathcal L ( B ), \quad \Delta (H) = \Delta (B), \quad \text{and} \quad
\Delta^* (H) = \Delta^* (B) \,.
\]

\medskip
\noindent{\bf Krull monoids.} A monoid $H$ is said to be a {\it Krull monoid} if it satisfies one of the following two equivalent conditions:
\begin{enumerate}
\item[(a)] There exists a monoid homomorphism $\varphi \colon H \to F$ into a free abelian monoid $F$ such that $a \t b$ in $H$ if and only if $\varphi (a) \t \varphi (b)$ in $F$.

\item[(b)] $H$ is completely integrally closed and $v$-noetherian.
\end{enumerate}
A detailed presentation of the theory of Krull monoids can be found in \cite{HK98, Ge-HK06a}. To recall some examples, note that an integral domain is a Krull domain if and only if its multiplicative monoid of nonzero elements is a Krull monoid. Thus Property (b) shows that every integrally closed noetherian domain is a Krull domain.
Rings of integers in algebraic number fields, holomorphy rings in algebraic function fields, and regular congruence monoids in these domains are Krull monoids with finite class group such that every class contains a prime divisor (\cite[Section 2.11 and Examples 7.4.2]{Ge-HK06a}). Monoid domains and power series domains that are Krull are discussed in \cite{Ki-Pa01, Ch11a}, and note that every class of a  Krull monoid domain contains a prime divisor. For monoids of modules that are Krull and their distribution of prime divisors, we refer the reader to \cite{Fa06a, Ba-Ge14b}.

\smallskip
Sets of lengths in Krull monoids can be studied in the monoid of zero-sum sequences over its class group. To recall the basic concepts, let
$G$ be an additive finite abelian group and $G_0 \subset G$ a subset.     An element $S = g_1 \cdot \ldots \cdot g_l \in \mathcal F (G_0)$ is called a {\it sequence} over $G_0$, $\sigma (S) = g_1+ \ldots+g_l$ denotes its sum, $\mathsf k (S)  = \sum_{i=1}^l \frac{1}{\ord (g_i)} \ \in \mathbb Q_{\ge 0}$ its {\it cross number} of $S$, $|S|=l$ its length, and
$\mathsf h(S)=\max \{\mathsf v_g(S)\mid g\in \supp(S)\}$ the maximal multiplicity of $S$.
Since the embedding
\[
\mathcal B (G_0) = \{S \in \mathcal F (G_0) \mid \sigma (S)=0 \} \hookrightarrow \mathcal F (G_0)
\]
satisfies Property (a) above, $\mathcal B (G_0)$ is a Krull monoid, called the {\it monoid of zero-sum sequences} over $G_0$.  Its significance for the study of general Krull monoids is summarized in the following lemma (see \cite[Theorem 3.4.10 and Proposition 4.3.13]{Ge-HK06a}).

\medskip
\begin{lemma} \label{2.1}
Let $H$ be a Krull monoid with finite class group $G$ such that every class contains a prime divisor. Then there is a transfer homomorphism $\theta \colon H \to \mathcal B (G)$. In particular, we have  $\mathcal L (H) = \mathcal L  \big( \mathcal B (G) \big)  $ and
$$
\Delta^* (H) = \Delta^* \big( \mathcal B (G) \big) = \big\{ \min \Delta \big( \mathcal B (G_0) \big) \mid G_0 \subset G \ \text{with} \ \Delta(\mathcal B (G_0))\neq \emptyset \big\} \,.
$$
\end{lemma}

Thus $\Delta^* (H)$ can be studied in an associated monoid of zero-sum sequences and can  be tackled by methods from Additive Combinatorics. The existence of a transfer homomorphism to a monoid of zero-sum sequences is not restricted to Krull monoids, but it holds true for so-called transfer Krull monoids and thus Theorem \ref{1.1} holds true for transfer Krull monoids over finite abelian groups. We refer to \cite{Ge16c} for a discussion of this concept and just mention one additional example.
Let $\mathcal O$ be a holomorphy ring in a global field $K$, $A$ a central simple algebra over $K$, and $H$ a classical maximal $\mathcal O$-order of $A$ such that every stably free left $R$-ideal is free. Then there is a transfer homomorphism from $H$ to the monoid of zero-sum sequences over a ray class group of $\mathcal O$ (\cite[Theorem 1.1]{Sm13a}).

\medskip
\noindent
{\bf Zero-Sum Theory.} Let $G$ be an additive finite abelian group and $G_0 \subset G$ a subset. We denote by $\langle G_0 \rangle \subset G$ the subgroup generated by $G_0$. Then $G\cong C_{n_1}\oplus \cdots \oplus C_{n_r}$, where $r=\mathsf{r}(G)\in \N_0$ is the \emph{rank} of $G$,  $n_r={\exp}(G)$ is the \emph{exponent} of $G$, and $1<n_1\mid \cdots \mid n_r\in \mathbb{N}$.
 It is traditional to set
$$
\mathcal A (G_0) := \mathcal A \big(\mathcal B (G_0)\big), \ \Delta (G_0):= \Delta \big( \mathcal B (G_0)\big) , \ \text{and} \   \Delta^* (G_0) :=  \Delta^* \big( \mathcal B (G_0)\big) \,.
$$
Clearly, the atoms of $\mathcal B (G_0)$ are precisely the minimal zero-sum sequences over $G_0$. The set $\mathcal A (G_0)$ is finite, and $\mathsf D (G_0) = \max \{ |S| \mid S \in \mathcal A (G_0)\}$ is the {\it Davenport constant} of $G_0$.
The set $G_0$ is called
\begin{itemize}
\item {\it half-factorial} \ if  the monoid $\mathcal B (G_0)$ is half-factorial (equivalently, $\Delta (G_0) = \emptyset$).

\item {\it non-half-factorial} \ if the monoid $\mathcal B (G_0)$ is not half-factorial (equivalently, $\Delta (G_0) \ne \emptyset$).

\item {\it minimal non-half-factorial} \ if $\Delta (G_0) \ne \emptyset$ but every proper subset is half-factorial.

\item an LCN-set if $\mathsf k (A) \ge 1$ for all $A \in \mathcal A (G_0)$.
\end{itemize}

The following simple result (\cite[Proposition 6.7.3]{Ge-HK06a}) will be used throughout the paper without further mention.

\medskip
\begin{lemma} \label{2.2}
Let $G$ be a finite abelian group and $G_0 \subset G$ a subset. Then
the following statements are equivalent{\rm \,:}
\begin{enumerate}
\item[(a)]
$G_0$ is half-factorial.
\smallskip

\item[(b)]
$\mathsf k (U) = 1$ for every $U \in \mathcal A (G_0)$.
\smallskip

\item[(c)]
$\mathsf L (B) = \{ \mathsf k (B) \}$ for every $B \in
      \mathcal B (G_0)$.
\end{enumerate}
\end{lemma}

We define    $$
\begin{aligned}
\mathsf m (G) & = \max \{ \min \Delta (G_0) \mid G_0 \subset G \ \text{is an LCN-set with} \ \Delta (G_0) \ne \emptyset \} \,,
\end{aligned}
$$
and we denote by $\Delta_1 (G)$ the set of all $d \in \N$ with the following property:
\begin{itemize}
\item[]  For every $k \in \N$, there exists some $L \in \mathcal L (G)$  which is an AAP (an almost arithmetical progression) with difference $d$ and length $l \ge k$.
\end{itemize}
Thus, by definition, if $G'$ is a further finite abelian group such that $\mathcal L (G) = \mathcal L (G')$, then $\Delta_1 (G) = \Delta_1 (G')$.
The next proposition gathers the  properties of $\Delta^* (G)$ and of $\Delta_1 (G)$ which are needed in the sequel.

\medskip
\begin{proposition} \label{2.3}
Let $G$ be a finite abelian group with $|G|\ge 3$ and $\exp (G)=n$.
\begin{enumerate}
\item $\Delta^* (G) \subset \Delta_1 (G) \subset \{ d_1 \in \Delta (G) \mid d_1 \ \text{divides some } \ d \in \Delta^* (G)\}$. In particular, $\max \Delta^* (G) = \max \Delta_1 (G)$.

\smallskip
\item $\max \Delta^* (G) = \max \{ \exp (G)-2, \mathsf m (G) \} = \max \{ \exp (G)-2, \mathsf r (G)-1\}$. If $G$ is a $p$-group, then $\mathsf m (G)=\mathsf r (G)-1$.

\smallskip
\item If $k\in \N$ is maximal such that $G$ has a subgroup isomorphic to $C_{n}^k$, then
      \[
      \Delta^*(G)\subset \Delta_1(G)\subset [1,\max \{ \mathsf m(G), \lfloor \frac{n}{2}\rfloor-1 \}]  \cup [\max \{ 1,n-k-1 \} ,n-2]\,.
      \]
      and
      $$
      [1,\mathsf r(G)-1]  \cup \{\max\{1,\left\lfloor \frac{n}{2}\right\rfloor-1\}\}\cup [\max \{ 1,n-k-1\} ,n-2] \subset \Delta^*(G)\subset \Delta_1(G) \,.
      $$
\end{enumerate}
\end{proposition}

\begin{proof}
1. follows from \cite[Corollary 4.3.16]{Ge-HK06a} and 2. from \cite[Theorem 1.1 and Proposition 3.2]{Ge-Zh16a}.

3. In  \cite[Theorem 3.2]{Sc09c}, it is proved that $\Delta^* (G)$ is contained in the set given above.  The set $[1, \mathsf r (G)-1] \cup  [\max \{ 1,n-k-1\} ,n-2]$ is contained in $\Delta^* (G)$ by \cite[Propositions 4.1.2 and 6.8.2]{Ge-HK06a} and $\{\max\{1,\left\lfloor \frac{n}{2}\right\rfloor-1\}\}$ is contained in $\Delta^* (G)$ by $|G|\ge 3$ and \cite[Theorem  6.8.12]{Ge-HK06a}.
\end{proof}

\medskip
\section{Minimal non-half-factorial subsets of finite abelian groups} \label{3}
\medskip

\centerline{\it Throughout this section, let $G$ be an additive finite abelian group}
\centerline{\it  with $|G| \ge 3$, $\exp (G)=n$, and $\mathsf r (G)=r$.}

\medskip
%We summarize the required machinery in four lemmas.
The first three lemmas gather basic properties of $\Delta^*(G)$ and of non-half-factorial sets.

\medskip
\begin{lemma} \label{3.1}
Let  $G_0 \subset G$  be a subset.
\begin{enumerate}
\item For each $g \in G_0$,
      \[
      \begin{aligned}
      &\gcd \big( \{ \mathsf v_g (B) \mid B \in \mathcal B (G_0) \} \big)  =  \gcd \big( \{ \mathsf v_g (A) \mid A \in \mathcal A (G_0) \} \big)\\ =&
      \min  \big( \{ \mathsf v_g (A) \mid \mathsf v_g (A)>0,  A  \in \mathcal A (G_0) \} \big)  \\
       =&  \min \big( \{ \mathsf v_g (B) \mid \mathsf v_g (B) > 0, B \in \mathcal B (G_0) \} \big) \\ =& \min  \big( \{ k \in \N  \mid kg \in \langle G_0 \setminus \{g\} \rangle  \} \big)=\gcd \big(  \{ k \in \N  \mid kg \in \langle G_0 \setminus \{g\} \rangle  \} \big) \,.
      \end{aligned}
      \]
In particular, $\min  \big( \{ k \in \N  \mid kg \in \langle G_0 \setminus \{g\} \rangle  \} \big)$ divides $\ord (g)$.

\smallskip
\item Suppose that for each two distinct elements $h, h' \in G_0$ we have  $h\not\in \langle G_0\setminus \{h,\ h'\}\rangle $.
 Then for any atom $A$ with $\supp(A)\subsetneq G_0$ and any $h \in \supp (A)$, we have $\gcd (\mathsf v_h(A), \ord(h) )>1$.

\smallskip
\item If $G_0$ is minimal non-half-factorial, then there exists a
   minimal  non-half-factorial subset $G_0^* \subset G$ with $|G_0|=|G_0^*|$ and a transfer homomorphism $\theta \colon \mathcal B (G_0) \to \mathcal B (G_0^*)$ such that  the following  properties are satisfied{\rm \,:}
    \begin{enumerate}
    \item   For each $g \in G_0^*$, we have $g \in \langle G_0^* \setminus \{g\} \rangle$.

    \smallskip
    \item  For each $B \in \mathcal B (G_0)$, we have $\mathsf k (B) = \mathsf k \big( \theta (B) \big)$.
  \smallskip
  \item If $G_0^*$ has the property that  for each $h\in G_0^*$, $h\not\in \langle E\rangle$ for any $E\subsetneq G_0^*\setminus \{h\}$, then $G_0$ also has the property.
    \end{enumerate}
\end{enumerate}
\end{lemma}

\begin{proof}
See \cite[Lemma 2.6]{Ge-Zh16a}.
\end{proof}

\medskip
\begin{lemma} \label{3.2}~

\begin{enumerate}
\item If $g \in G$ with $\ord (g) \ge 3$, then $\ord (g) - 2 \in \Delta^* (G)$. In particular, $n-2 \in \Delta^* (G)$.

\smallskip
\item If $r \ge 2$, then $[1, r -1] \subset \Delta^* (G)$.

\smallskip
\item Let $G_0 \subset G$ be a subset.
      \begin{enumerate}
      \item If there exists an $U \in \mathcal A (G_0)$ with $\mathsf k (U) < 1$, then $\min \Delta (G_0) \le \exp (G)-2$.

      \smallskip
      \item If $G_0$ is an \text{\rm LCN}-set, then $\min \Delta (G_0) \le |G_0|-2$.
      \end{enumerate}
\end{enumerate}
\end{lemma}

\begin{proof}
See \cite[Proposition 6.8.2 and Lemmas 6.8.5 and 6.8.6]{Ge-HK06a}.
\end{proof}

\medskip
\begin{lemma}\label{3.5}
Let $G_0\subset G$ be a non-half-factorial subset satisfying the following two conditions{\rm \,:}
\begin{itemize}
\item[(a)] There is some $g\in G_0$ such that $\Delta (G_0\setminus \{g\})=\emptyset$.
\item[(b)] There is some  $U\in \mathcal A(G_0)$ with $\mathsf k(U)=1$ and $\gcd(\mathsf v_g(U),\ord(g))=1$.
\end{itemize}
Then $\mathsf k(\mathcal A(G_0))\subset \N$ and
\[
\min \Delta (G_0)\mid \gcd \{\mathsf k(A)-1 \mid A\in \mathcal A(G_0)\}\,.
\]
Note that the conditions hold if $\Delta(G_1)=\emptyset$ for each $G_1\subsetneq G_0$ and there exists some $G_2$ such that $\langle G_2\rangle=\langle G_0\rangle$ and $|G_2|\le |G_0|-2$.
\end{lemma}

\begin{proof}
The first statement follows from \cite[Lemma 6.8.5]{Ge-HK06a}.  If $\Delta (G_1) = \emptyset$ for all $G_1 \subsetneq G_0$, then Condition (a) holds. Let $G_2 \subsetneq G_1 \subsetneq G_0$ with $\langle G_2 \rangle = \langle G_0 \rangle$. If $g \in G_1 \setminus G_2$, then $\langle G_2 \rangle = \langle G_0 \rangle$ implies that there is some $U \in \mathcal A (G_1)$ with $\mathsf v_g (U)=1$, and since $G_1 \subsetneq G_0$, it follows that $\mathsf k (U)=1$.
\end{proof}

\medskip
\begin{lemma} \label{3.3}
Let  $G_0\subset G$ be a subset,  $g\in G_0\setminus\{0\}$, and  $g\in \langle G_0\setminus\{g\}\rangle$. Then  for each prime $p$ dividing $\ord(g)$, there exists an atom $A\in \mathcal A(G_0)$ with $2\le |\supp(A)|\le r+1$,  $\mathsf v_g(A) \le  \ord(g)/2$, $\mathsf v_g(A) \t \ord(g)$, and $p\nmid \mathsf v_g(A)$. In particular,
\begin{enumerate}
\item If $|G_0|\ge r+2$, then there exist $s_0<\ord(g)$ and $E\subsetneq G_0\setminus \{g\} $ such that $s_0g\in \langle E \rangle$.

\item If $\ord(g)$ is a prime power, then there exists a subset $E\subset G_0\setminus\{g\}$ with $|E|\le r$ such that $g\in \langle E\rangle $.
\end{enumerate}
\end{lemma}

\begin{proof}
 We set $\exp(G)=n = p_1^{k_1} \cdot \ldots \cdot p_t^{k_t}$, where $t, k_1, \ldots, k_t \in \N$ and $p_1, \ldots, p_t$ are distinct primes. 
  Let $\nu \in [1,t]$ with $p_{\nu} \t \ord (g)$.
Since $g\in \langle G_0\setminus \{g\} \rangle$, it follows that $0 \ne \frac{n}{p_{\nu}^{k_{\nu}}}g\in G_{\nu} = \langle \frac{n}{p_{\nu}^{k_{\nu}}}h \mid h\in G_0\setminus \{g\}\rangle$. Obviously, $G_{\nu}$ is a $p_{\nu}$-group.
Let $E_{\nu} \subset G_0\setminus \{g\}$ be minimal such that $\frac{n}{p_{\nu}^{k_{\nu}}}g \in  \langle \frac{n}{p_{\nu}^{k_{\nu}}} E_{\nu}\rangle$. 
The minimality of $E_{\nu}$ implies that $|E_{\nu}| = |\frac{n}{p_{\nu}^{k_{\nu}}} E_{\nu}|$ and it implies that $\frac{n}{p_{\nu}^{k_\nu}} E_{\nu}$ is a minimal generating set of $G_{\nu}' := \langle \frac{n}{p_{\nu}^{k_{\nu}}} E_{\nu} \rangle$. Thus \cite[Lemma A.6.2]{Ge-HK06a} implies that $|\frac{n}{p_{\nu}^{k_\nu}} E_{\nu}| \le \mathsf r^* (G_{\nu}')=\mathsf r(G_{\nu}')\le \mathsf r(G_{\nu})$ (note that $\mathsf r^* (G_{\nu}')$ is the total rank of $G_{\nu}'$).  Putting all together we obtain that
\[
1\le |E_{\nu}| = |\frac{n}{p_{\nu}^{k_{\nu}}} E_{\nu}| \le \mathsf r (G_{\nu})  \le r \,.
\]
Let $d_{\nu} \in \N$ be  minimal  such that $d_{\nu}g\in \langle E_{\nu}\rangle$. Since
$0 \ne \frac{n}{p_{\nu}^{k_{\nu}}}g\in \langle E_{\nu} \rangle$, it follows that $d_{\nu} < \ord (g)$.
By Lemma \ref{3.1}.1, $d_{\nu}\t \gcd(\frac{n}{p_{\nu}^{k_{\nu}}}, \ord(g))$ and there exists an atom $U_{\nu}$ such that $\mathsf v_g(U_{\nu})=d_{\nu}$ and $|\supp(U_{\nu})\setminus\{g\}| \le |E_{\nu}| \le r$. Therefore $|\supp (U_{\nu})| \le r+1$, $d_{\nu}\t \ord(g)$, and $p_{\nu}\nmid d_{\nu}$. Since $p_{\nu}\t \ord (g)$, it follows that  $d_{\nu}\le \ord(g)/2$ and $|\supp (U_{\nu})| \ge 2$. 

If $|G_0|\ge r+2$, then $|E_{\nu}|\le r<|G_0\setminus \{g\}|$ implies that $E_{\nu} \subsetneq G_0 \setminus \{g\}$, and the assertion holds with $E=E_{\nu}$ and $s_0 = d_{\nu}$.

If $\ord(g)$ is a prime power, then $\ord(g)$ is a power of $p_{\nu}$ which implies that  $\gcd(\frac{n}{p_{\nu}^{k_{\nu}}}, \ord(g))=1$ whence $d_{\nu}=1$ and $g\in \langle E_{\nu}\rangle$.
\end{proof}

\medskip
\begin{lemma} \label{3.4}
Let $G_0 \subset G$ be a minimal non-half-factorial {\rm LCN}-set with $|G_0|\ge r+2$ such that $h \in \langle G_0 \setminus \{h\} \rangle$ for every $h \in G_0$.  Suppose that
for each two distinct elements $h, h' \in G_0$,  we have $h\not\in \langle G_0\setminus \{h,\ h'\}\rangle $,  and each atom $A\in \mathcal A(G_0)$ with $\supp(A)=G_0$ has cross number $\mathsf k(A)>1$.
Then $\min \Delta(G_0)\le \lfloor\frac{n}{2}\rfloor-1$.
\end{lemma}

\begin{proof}
We choose an element $g\in G_0$. If $\ord(g)$ is a prime power, then there exists $E\subset G_0\setminus \{g\}$ such that $g\in \langle E\rangle$ and $|E|\le r<|G_0|-1$ by Lemma \ref{3.3}.2, a contradiction to the assumption on $G_0$. Thus $\ord(g)$ is not a prime power.

Let $s \in \N$ be  minimal  such that there exists a subset $E\subsetneq G_0\setminus \{g\}$ with $sg\in \langle E \rangle$, and  by Lemma \ref{3.3}.1, we observe that $s < \ord (g)$. Let  $E \subsetneq G_0\setminus \{g\}$ be  minimal such that $sg\in \langle E \rangle$. By Lemma \ref{3.1}.1, there is an atom   $V$  with $\mathsf v_g(V)=s\t \ord(g)$ and $\supp(V)=\{g\}\cup E\subsetneq G_0$.
By Lemma \ref{3.1}.2, for each $h\in \supp(V)$, $\mathsf v_h(V)\ge 2$ which implies that
 $s \ge 2$. Thus there is a  prime  $p\in \N$   dividing $s$ and hence $p \t s \t \ord (g)$.  By Lemma \ref{3.3}, there exists an atom $U_1$ such that $|\supp(U_1)|\le r+1$, $\mathsf v_g(U_1)\t \ord(g)$, and  $p\nmid  \mathsf v_g(U_1)$, and therefore $\supp(U_1)\subsetneq G_0$.

Let $d=\gcd(s,\mathsf v_g(U_1))$. Then $d<s<\mathsf v_g(U_1)$ and
there exist $x_1\in [1,\frac{\ord(g)}{s}-1]$ and $x_2\in [1,\frac{\ord(g)}{\mathsf v_g(U_1)}-1]$ such that  $d+\ord(g)=x_1s+x_2\mathsf v_g(U_1)$. Let $V^{x_1}U_1^{x_2}=g^{\ord(g)}\cdot W$, where $W \in \mathcal B (G_0)$  with $\mathsf v_g(W)=d$, and let $W_1$ be an atom dividing $W$ with $\mathsf v_g(W_1)>0$. Since $\mathsf v_g(W_1)\le d<s$, the minimality of $s$ implies that $\supp(W_1)=G_0$ and hence  $\mathsf k(W_1)>1$. Since $G_0$ is minimal non-half-factorial, we have that $\mathsf k(V)=\mathsf k(U_1)=1$. Therefore there exists $l\in \N$ with $2\le l<x_1+x_2$ such that $\{l, x_1+x_2\}\subset \mathsf L(V^{x_1}U_1^{x_2})$.
Let $W=X_1\cdot \ldots \cdot X_{x_1+x_2}$ and $g^{\ord(g)}=g^{y_1}\cdot\ldots\cdot g^{y_{x_1+x_2}}$ such that $X_ig^{y_i}=V$ for each $i\in [1,x_1]$ and $X_ig^{y_i}=U_1$ for each $i\in [x_1+1, x_1+x_2]$, where $X_1,\ldots, X_{x_1+x_2}\in \mathcal F(G_0)$ and  $y_1,\ldots,y_{x_1+x_2}\in \N$.  If there exist distinct  $i,j\in [1,x_1+x_2]$ such that $y_i=y_j=1$, then $2\mathsf v_g(W)+2=2d+2\le \mathsf v_g(X_ig^{y_i}X_jg^{y_j})\le y_i+y_j+\mathsf v_g(W)$ which implies that $y+i+y_j\ge \mathsf v_g(W)+2\ge 3$, a contradiction.
 Therefore $|\{i\in [1,x_1+x_2]\mid y_i=1 \}|\le 1$. It follows that $1+2(x_1+x_2-1)\le \ord(g)$.
 Then
\[
\min \Delta(G_0)\le x_1+x_2-l\le \frac{\ord(g)+1}{2}-2\le \left\lfloor\frac{n}{2}\right\rfloor-1 \,.  \qedhere
\]
\end{proof}

\medskip
\begin{lemma} \label{3.6}
Let $G_0 \subset G$ be a minimal non-half-factorial {\rm LCN}-set with $|G_0|\ge r+2$ such that $h \in \langle G_0 \setminus \{h\} \rangle$ for every $h \in G_0$. Suppose that one of the following properties is  satisfied{\rm \,:}
\begin{enumerate}
\item[(a)] For each two distinct elements $h, h' \in G_0$,  we have $h\not\in \langle G_0\setminus \{h,\ h'\}\rangle $, and
           there is an atom $A\in \mathcal A(G_0)$ with  $\mathsf k(A)=1$  and $\supp(A)=G_0$.

\smallskip
\item[(b)] There is a subset $G_2 \subset G_0$ such that $\langle G_2 \rangle = \langle G_0 \rangle$ and $|G_2| \le |G_0|-2$.
\end{enumerate}
Then $\min \Delta (G_0) \le \max\{r-1, \left \lfloor \frac{n}{2}\right\rfloor-1 \}  $.
\end{lemma}

\begin{proof}
Assume to the contrary that $\min \Delta (G_0) \ge  \max\{r, \left \lfloor \frac{n}{2}\right\rfloor \}$. Then Lemma \ref{3.2}.3.(b) implies that $|G_0|\ge 2+\min\Delta(G_0)\ge \frac{n}{2}+1$. If  Property (a) is satisfied, then  there exists some $g\in G_0$ such that $\mathsf v_g(A)=1$.
By Lemma \ref{3.5}, each of the two Properties (a) and (b) implies that
 $\mathsf k (U) \in \N$ for each $U \in \mathcal A (G_0)$ and
\[
\min \Delta (G_0) \t \gcd \big( \{ \mathsf k (U) - 1 \mid U \in \mathcal A (G_0) \} \big) \,.
\]
We set
\[
\Omega_{=1}=\{A\in \mathcal A (G_0) \mid \mathsf k (A)=1\} \quad \text{and} \quad \Omega_{>1}=\{A\in \mathcal A (G_0) \mid \mathsf k (A)>1\} \,.
\]
Thus for each $U_1, U_2 \in \Omega_{>1}$ we have
\begin{equation} \label{e2}
\begin{aligned}
&\mathsf k (U_1)\ge \max\{r+1, \left \lfloor \frac{n}{2}\right\rfloor+1 \} \quad \text{and}\\
& \quad \big(\text{either} \ \mathsf k (U_1)=\mathsf k (U_2) \ \text{or} \ |\mathsf k (U_1)-\mathsf k (U_2)|\ge  \max\{r, \left \lfloor \frac{n}{2}\right\rfloor \} \big) \,.
\end{aligned}
\end{equation}
Furthermore, for each $U \in \Omega_{=1}$ we have $\mathsf h (U) \ge 2$ (otherwise, $U$ would divide every atom $U_1 \in \Omega_{>1}$).
We claim that

\begin{enumerate}
\item[{\bf  A1.}] For each $U\in \Omega_{>1}$,  there are $A_1,  \ldots , A_m \in \Omega_{=1}$,
                  where $m \le \frac{n+1}{2}$,  such that $UA_1 \cdot \ldots \cdot A_m$ can be factorized into a product of atoms from $\Omega_{=1}$.
\end{enumerate}

\medskip

{\it Proof of {\bf  A1.}} \, Suppose that  Property (a) holds. As observed above there exists some $g\in G_0$ such that $\mathsf v_g(A)=1$.  Lemma \ref{3.3} implies that there is an atom $X$ such that
$2\le |\supp(X)|\le \mathsf r(G)+1$ and $1 \le  \mathsf v_g(X) \le  \ord(g)/2$. Since  $g\not\in \langle G_0\setminus \{g,\ h\}\rangle$ for any $h\in G_0\setminus \{g\}$, it follows that $\mathsf v_g (X) \ge 2$, and  $|G_0|\ge  r+2$ implies  $\supp(X)\subsetneq G_0$.

Suppose that Property (b) is satisfied.
We choose an element $g\in G_0\setminus G_2$. Then $g\in \langle G_2 \rangle$ and by Lemma \ref{3.1}.1,  there is an atom $A'$ with
$\mathsf v_g(A')=1$ and $\supp(A')\subset G_2 \cup \{g\}\subsetneq G_0$. This implies that  $A'\in \Omega_{=1}$. Let $h \in G_0$ such that  $\mathsf v_h(A')=\mathsf h(A')$. Since $\mathsf h(A')\ge 2$, we obtain that
$A'^{\lceil\frac{\ord(h)}{\mathsf h(A')}\rceil}=h^{\ord(h)}\cdot W$
where $W$ is a product of $\lceil\frac{\ord(h)}{\mathsf h(A')}\rceil-1$ atoms and $\mathsf v_g(W)=\lceil\frac{\ord(h)}{\mathsf h(A')}\rceil$.
Thus there exists an atom $X'$ with $2\le\mathsf v_g(X')\le \lceil\frac{\ord(h)}{\mathsf h(A')}\rceil\le \frac{n}{2}+1$.

\smallskip
Therefore both properties imply that there are  $A, X \in \mathcal A (G_0)$ and  $g\in G_0$ such that $\mathsf k(A)=\mathsf k(X)=1$, $\mathsf v_{g}(A)=1$,  and $2\le\mathsf v_g(X)\le \frac{n}{2}+1$. Let $U\in \Omega_{>1}$.
\smallskip

If $\ord(g)-\mathsf v_g(U)< \mathsf v_g(X)\le \frac{n}{2}+1$, then
\[
U  A^{\ord(g)-\mathsf v_g(U)}=g^{\ord(g)}  S\,,
\]
where $S \in \mathcal B (G_0)$  and $\ord(g)-\mathsf v_g(U)\le \frac{n}{2}$. Since $\supp (S) \subsetneq G_0$, $S$ is a product of atoms from $\Omega_{=1}$.

If $\ord(g)-\mathsf v_g(U)\ge \mathsf v_g(X)$, then
\[
U X^{\lfloor \frac{\ord(g)-\mathsf v_g(U)}{\mathsf v_g(X)}\rfloor} A^{\ord(g)-\mathsf v_g(U)-\mathsf v_g(X)\cdot \lfloor \frac{\ord(g)-\mathsf v_g(U)}{\mathsf v_g(X)}\rfloor }=g^{\ord(g)}  S\,,
\]
where $S$ is a product of atoms from $\Omega_{=1}$ (because $\supp (S) \subsetneq G_0)$ and
\[
\begin{aligned}
\lfloor  & \frac{\ord(g)-\mathsf v_g(U)}{\mathsf v_g(X)}\rfloor+\ord(g)-\mathsf v_g(U)-\mathsf v_g(X)\cdot \lfloor \frac{\ord(g)-\mathsf v_g(U)}{\mathsf v_g(X)}\rfloor \\
& \le \frac{\big( \ord (g) - \mathsf v_g (U)\big) - \big(\mathsf v_g (X)-1 \big)}{\mathsf v_g (X)} + \mathsf v_g (X)-1 \\
& \le \frac{\ord (g)-\mathsf v_g (U) + 1}{2} \le \frac{n+1}{2}\,. \qquad \qquad \qquad  \qed{\text{\rm (Proof of {\bf  A1})}}
\end{aligned}
\]

\smallskip
We set
\[
\Omega_{>1}'=\{A\in \mathcal A (G_0) \mid \mathsf k (A)= \min \{ \mathsf k (B) \mid B \in \Omega_{>1}\} \}\subset \Omega_{>1} \,,
\]
and we consider all tuples $(U,A_1,\ldots,A_m)$, where $U\in \Omega_{>1}'$, $m\in \N$, and $A_1,\ldots,A_m\in \Omega_{=1}$,  such that $UA_1 \cdot \ldots \cdot A_m$ can be factorized into a product of atoms from $\Omega_{=1}$. We fix one such tuple $(U,A_1,\ldots,A_m)$ with the property that $m$ is minimal possible.
Let
\begin{equation}\label{e3}
UA_1 \cdot \ldots \cdot A_m = V_1 \cdot \ldots \cdot V_t \quad \text{ with} \quad t \in \N \quad \text{and} \quad  V_1, \ldots, V_t \in \Omega_{=1} \,.
\end{equation}
We observe that  $\mathsf k(U)=t-m$ and continue with the following assertion.

\begin{enumerate}
\item[{\bf  A2. }]  For each $\nu \in [1, t]$, we have $V_{\nu} \nmid UA_1 \cdot \ldots \cdot A_{m-1}$.
\end{enumerate}

{\it Proof of {\bf A2.}} \, Assume to the contrary that there is such a $\nu \in [1, t]$, say $\nu = 1$, with  $V_1 \t U A_1 \cdot \ldots \cdot A_{m-1}$. Then there are $l \in \N$ and $T_1, \ldots, T_l \in \mathcal A (G_0)$ such that
\[
U A_1 \cdot \ldots \cdot A_{m-1} = V_1 T_1 \cdot \ldots \cdot T_{\mathit l} \,.
\]
By the minimality of $m$,  there exists some $\nu \in [1, l]$ such that $T_{\nu} \in \Omega_{>1}$, say $\nu=1$.  Since
\[
\sum_{\nu=2}^l \mathsf k (T_{\nu}) = \mathsf k (U) + (m-1)-1 - \mathsf k (T_1) \le m-2 \le \frac{n-3}{2}\,,
\]
and $\mathsf k (T')\ge\frac{n}{2}$ for all $T' \in \Omega_{>1}$, it follows that  $T_2, \ldots, T_l \in \Omega_{=1}$, whence $l = 1 + \sum_{\nu=2}^l \mathsf k (T_{\nu}) \le m-1$.
We obtain that
\[
V_1 T_1 \cdot \ldots \cdot T_{\mathit l}A_m = U A_1 \cdot \ldots \cdot A_m = V_1 \cdot \ldots \cdot V_t \,,
\] and thus
 \[
 T_1 \cdot \ldots \cdot T_{\mathit l}A_m = V_2 \cdot \ldots \cdot V_t \,.
\]
The minimality of $m$ implies that $\mathsf k(T_1)> \mathsf k(U)$.
It follows that  \[\mathsf k(T_1)-\mathsf k(U)=m-1-{\mathit l}\le m-2\le \frac{n-3}{2}<\left\lfloor\frac{n}{2}\right\rfloor\le \mathsf k(T_1)-\mathsf k(U)\,,\]
 a contradiction.  \qed{(Proof of {\bf  A2})}

\bigskip
\medskip
With the minimal integer $m$, as fixed before {\bf A2}, we consider all the tuples $(A_1',\ldots,A_m')$, where $A_1',\ldots,A_m'\in \Omega_{=1}$, such that $UA_1'\cdot\ldots\cdot A_m'$ can be factorized into a product of atoms from $\Omega_{=1}$. We fix one such tuple $(A_1',\ldots,A_m')$ such that $|\supp(A_m')|$ is minimal. For simplicity of notation, we suppose that  $(A_1',\ldots,A_m')=(A_1,\ldots,A_m)$.

\medskip
By Equation (\ref{e3}), there are $X_1, Y_1, \ldots,  X_t, Y_t \in \mathcal F (G)$ such that
\begin{align*}
&UA_1\cdot \ldots \cdot A_{m-1}=X_1\cdot\ldots\cdot X_t, \\ &A_{m}=Y_1\cdot\ldots\cdot Y_t, \text{ and } \ V_i=X_iY_i \text{ for each } \ i\in[1,t] \,.
\end{align*}
Then {\bf A2} implies that $|Y_i| \ge 1$ for each $i \in [1,t]$, and we set
$\alpha=|\{i\in [1,t]\mid |Y_i|=1\}|$. If $\alpha\le 2m$, then
\[
n\ge |A_m|=|Y_1|+\ldots+|Y_t|\ge \alpha+2(t-\alpha)=2t-\alpha\ge 2t-2m \,,
\]
 and hence $\min \Delta (G_0)\le t-1-m\le \frac{n}{2}-1$, a contradiction. Thus
$\alpha\ge 2m+1$.
After renumbering if necessary we  assume that $1=|Y_1|=\ldots=|Y_{\alpha}|<|Y_{\alpha+1}|\le \ldots\le |Y_t|$.
Let $Y_i=y_i$ for each $i\in[1,\alpha]$ and
\begin{equation}\label{e6}
S_0=\{y_1,y_2,\ldots,y_{\alpha}\}\,.
\end{equation}
For every $i \in [1, \alpha]$, $V_i \t y_iU A_1 \cdot \ldots \cdot A_{m-1}$ whence $\mathsf v_{y_i} (V_i) \le 1 + \mathsf v_{y_i} (U A_1 \cdot \ldots \cdot A_{m-1})$ and since $V_i \nmid U A_1 \cdot \ldots \cdot A_{m-1}$, it follows that
\begin{equation} \label{e7}
\mathsf v_{y_i}(V_i)=\mathsf v_{y_i}(UA_1\cdot\ldots\cdot A_{m-1})+1 \,.
\end{equation}
Assume to the contrary that there are distinct $i,j \in [1, \alpha]$ such that $y_i=y_j$. Then
\[
\mathsf v_{y_i}(UA_1\ \cdot \ldots \cdot  A_{m-1})+1=\mathsf v_{y_i}(V_i)=\mathsf v_{y_i}(X_i)+1 =\mathsf v_{y_i}(V_j)=\mathsf v_{y_i}(X_j)+1 \,.
\]
Since $X_iX_j \t UA_1\ \cdot \ldots \cdot  A_{m-1}$, we infer that
\[
\mathsf v_{y_i}(UA_1\cdot\ldots\cdot A_{m-1}) \ge \mathsf v_{y_i}(X_iX_j)=\mathsf v_{y_i}(V_iV_j)-2 = 2\mathsf v_{y_i}(UA_1\cdot\ldots\cdot A_{m-1}) \,,
\]
which implies that  $\mathsf v_{y_i}(UA_1\ldots A_{m-1})=0$, a contradiction to  $\supp(U)=G_0$. Thus $|S_0|=\alpha$ and
\begin{equation}\label{e5}
|\supp(A_m)|\ge |S_0|=\alpha\ge 2m+1\,.
\end{equation}

\medskip

 We proceed by the following  assertion.

\begin{enumerate}
\item[{\bf  A3. }]  $|\supp(A_m)|\le r+1$.
\end{enumerate}

{\it Proof of {\bf A3.}} Assume to the contrary that $|\supp(A_m)|\ge r+2$. We fix one element $g'\in S_0$.
Let $s_0 \in \N$  be minimal such that there exists a subset $E \subsetneq \supp(A_m)\setminus\{g'\}$ such that $s_0 g'\in \langle E\rangle$. By $|\supp (A_m)| \ge  r+2$, Lemma \ref{3.3} (applied to the subset $\supp (A_m) \subset G_0$)  implies that $s_0 < \ord (g')$.
Let $E$ be a minimal subset with this property.
Thus, by Lemma \ref{3.1}.1,  there exists an atom $A'$ with $\mathsf v_{g'} (A')=s_0$ and $\supp(A')=\{g'\}\cup E\subsetneq \supp(A_m)\subset G_0$
which implies that $\mathsf k (A')=1$.

 If $s_0=1$, then we assume that $g'=y_1$. Since $\mathsf v_{y_1}(V_1)=\mathsf v_{y_1}(UA_1\cdot\ldots\cdot A_{m-1})+1$ by Equation \ref{e7} and $V_1\t UA_1\cdot\ldots\cdot A_{m-1}\cdot y_1$, we obtain that $|\supp(UA_1\cdot\ldots\cdot A_{m-1}\cdot A'(V_1)^{-1})|<|G_0|$ and hence $UA_1\cdot\ldots\cdot A_{m-1}\cdot A'$ can be factorized into a product of atoms from $\Omega_{=1}$, a contradiction to the minimality of $|\supp(A_m)|$.

Suppose  $s_0\ge 2$.  We distinguish two cases:

\medskip
\noindent
CASE 1: \, $|\supp(A')\cap S_0|\ge m+1$.
\medskip

We may suppose  that $\{y_1,\ldots,y_{m+1}\}\subset \supp(A')\cap S_0$.
Then $V_1\cdot\ldots\cdot V_{m+1}\t UA_1\cdot\ldots\cdot A_{m-1}A'$ and $\mathsf k (UA_1\cdot\ldots\cdot A_{m-1}A'(V_1\cdot\ldots\cdot V_{m+1})^{-1})<\mathsf k(U)$. By the minimality of $\mathsf k(U)$, we have that $UA_1\cdot\ldots\cdot A_{m-1}A'$ can be factorized into a product of atoms from $\Omega_{=1}$, a contradiction to the minimality of $|\supp(A_m)|$.

\medskip
\noindent
CASE 2: \, $|\supp(A')\cap S_0|\le m$.
\medskip

 Let $p$ be a prime dividing $s_0$.
 Lemma \ref{3.3} (applied to the subset $\supp (A_m) \subset G_0$) implies  that  there exists an atom $A_p'\in \mathcal A(\supp(A_m))$ such that $|\supp(A_p')|\le r+1<|\supp(A_m)|$ and $p\nmid \mathsf v_{g'}(A_p')$.

Let $d=\gcd(s_0,\mathsf v_{g'}(A_{p}')$. Then $d<s_0$ and
\[
dg'\in \langle s_0 g', \mathsf v_{g'}(A_{p}')g'\rangle\subset \langle (\supp(A')\cup \supp(A_{p}') ) \setminus \{g'\}\rangle \,.
\]

Thus by minimality of $s_0$, we have $\supp(A_m)\setminus\{g'\}= \Big(\supp(A'')\cup \supp(A_{p}') \Big)\setminus \{g'\}$.
It follows that
\begin{align*}
|\supp(A_p')\cap S_0|&\ge |S_0\setminus \supp(A')|\ge |S_0|-|\supp(A')\cap S_0|\\
&\ge 2m+1-m=m+1\,.
\end{align*}
Similar to CASE 1,  $UA_1\cdot\ldots\cdot A_{m-1}A_p'$ can be factorized into a product of atoms from $\Omega_{=1}$, a contradiction to the minimality of $|\supp(A_m)|$.
\qed{(Proof of {\bf A3})}

\bigskip
\medskip

We consider all tuples $T=(X_1, Y_1, \ldots,  X_t, Y_t)$, where  $X_1, Y_1, \ldots,  X_t, Y_t\in \mathcal F (G)$, such that
\begin{align*}
&UA_1\cdot \ldots \cdot A_{m-1}=X_1\cdot\ldots\cdot X_t, \\ &A_{m}=Y_1\cdot\ldots\cdot Y_t, \text{ and } \ V_i=X_iY_i \text{ for each } \ i\in[1,t] \,.
\end{align*}
After renumbering if necessary, we can assume that   $|Y_i|=1$ for each $i\in [1,s_1]$, $|Y_i|=2$ and $\supp(Y_i)=1$ for each $i\in [s_1+1,s_2]$, $|Y_i|=2$ and $\supp(Y_i)=2$ for each $i\in [s_2+1,s_3]$, and $|Y_i|\ge 3$ for each $i\in [s_3+1,t]$, where $s_1,s_2,s_3\in [0,t]$.
Let $F_1(T)=\supp(Y_1\cdot\ldots\cdot Y_{s_1})$, $F_2(T)=\supp(Y_{s_1+1}\cdot\ldots\cdot Y_{s_2})$,
$F_3(T)=\supp(Y_{s_2+1}\cdot\ldots\cdot Y_{s_3})$, and
$F_4(T)=\supp(Y_{s_3+1}\cdot\ldots\cdot Y_{t})$.

Now we fix one such tuple $T=(X_1, Y_1, \ldots,  X_t, Y_t)$ such that $\big(\alpha_T=|\{i\in [1,t]\mid |Y_i|=1\}|, |F_1(T)\cap F_3(T)|\big)\in (\N_0^2,+)$ is minimal with respect to  lexicographic order.

\bigskip

\begin{enumerate}
\item[{\bf  A4.}] There exits a subset $\{g_1,\ldots,g_{\ell}\}\subset \supp(A_m)$ with $\ell\le r-m$  such that \\
$UA_1 \cdot \ldots \cdot A_{m-1}g_1^{\ord{g_1}}\cdot\ldots\cdot g_{\ell}^{\ord(g_{\ell})} $ can be factorized into a product of atoms from $\Omega_{=1}$.
\end{enumerate}

\medskip
{\it Proof of {\bf A4.}}
If $F_1(T)\cap F_4(T)\neq \emptyset$, there exist $i\in [1,s_1]$ and $j\in [s_3+1,t]$ such that $Y_i\cap Y_j=\{y_i\}$, where $Y_i=\{y_i\}$. By Equation \eqref{3.4}, $\mathsf v_{y_i}(X_i)\ge 1$. Let $X_i'=X_iy_i^{-1}, Y_i'=Y_iy_i, X_j'=X_jy_i, Y_j'=Y_jy_i^{-1}$ and substitute $X_i,Y_i,X_j,Y_j$ with $X_i',Y_i',X_j',Y_j'$  in the tuple $T=(X_1, Y_1, \ldots,  X_t, Y_t)$. Thus we get a new tuple $T'$ such that $\alpha_{T'}=\alpha_T-1$, a contradiction  to  the minimality of $\alpha_T$. Thus $F_1(T)\cap F_4(T)=\emptyset$.

If $F_1(T)\cap F_3(T)\neq \emptyset$, there exist $i\in [1,s_1]$ and $j\in [s_2+1,s_3]$ such that $Y_i\cap Y_j=\{y_i\}$, where $Y_i=\{y_i\}$.  Let $Y_j=\{y_i,y_j\}$, where $y_j\neq y_i$. By Equation \eqref{3.4}, $\mathsf v_{y_i}(X_i)\ge 1$. Let $X_i'=X_iy_i^{-1}, Y_i'=Y_iy_i, X_j'=X_jy_i, Y_j'=Y_jy_i^{-1}$ and substitute $X_i,Y_i,X_j,Y_j$ with $X_i',Y_i',X_j',Y_j'$  in the tuple $T=(X_1, Y_1, \ldots,  X_t, Y_t)$. Thus we get a new tuple $T'$ such that $\alpha_{T'}=\alpha_T$, $|F_1(T')\cap F_3(T')|=|F_1(T)\cap F_3(T)|-1$, a contradiction  to  the minimality of $\big(\alpha_T=|\{i\in [1,t]\mid |Y_i|=1\}|, |F_1(T)\cap F_3(T)|\big)$. Thus $F_1(T)\cap F_3(T)=\emptyset$.

Suppose that  $|F_1(T)\cap F_2(T)|\ge m$. Then let $\{g_1,\ldots,g_m\}\subset F_1(T)\cap F_2(T)$ and $Y_i=g_i$, $Y_{s_1+i}=g_i^2$, for each $i\in [1,m]$.    Hence $$\prod_{i\in [1,m]}(V_iV_{s_1+i})\mid UA_1\cdot\ldots A_{m-1}g_1^{\ord(g_1)}\cdot\ldots\cdot g_{m}^{\ord(g_m)}\,,$$ and
$$\mathsf k\big(UA_1\cdot\ldots A_{m-1}g_1^{\ord(g_1)}\cdot\ldots\cdot g_{m}^{\ord(g_m)}(\prod_{i\in [1,m]}(V_iV_{s_1+i}))^{-1}\big)= \mathsf k(U)-1\,.$$
It follows by the minimality of $\mathsf k(U)$ that $ UA_1\cdot\ldots A_{m-1}g_1^{\ord(g_1)}\cdot\ldots\cdot g_{m}^{\ord(g_m)}$ can be factorized into a product of atoms from $\Omega_{=1}$. Note that $r+1\ge |\supp(A_m)|\ge 2m+1$ by {\bf A3} and Equation \eqref{3.5}. We have that $\ell=m\le r-m $.

Suppose that $|F_1(T)\cap F_2(T)|\le  m-1$. Then $|F_1(T)\setminus F_2(T)|\ge m+2$. Since $F_1(T)\cap F_4(T)=\emptyset$ and $F_1(T)\cap F_3(T)=\emptyset$, we  let $\{g_1,\ldots, g_{m+2}\}\subset F_1(T)\setminus (F_2(T)\cup F_3(T)\cup F_4(T))$ and $\supp(A_m)\setminus \{g_1,\ldots,g_{m+2}\}=\{h_1,\ldots, h_\ell\}$, where $\ell\le r-1-m$. We assume that $Y_i=g_i$ for each $i\in [1,m+2]$.  Therefore $$\prod_{i\in [m+3,t]}V_i\mid UA_1\cdot\ldots A_{m-1}h_1^{\ord(h_1)}\cdot\ldots\cdot h_{\ell}^{\ord(h_{\ell})}\,,$$ and
$$\mathsf k\big(UA_1\cdot\ldots A_{m-1}h_1^{\ord(h_1)}\cdot\ldots\cdot h_{\ell}^{\ord(h_{\ell})}(\prod_{i\in [m+3,t]}V_i)^{-1}\big)=\mathsf k(U)+m-1+\ell-(t-m-2)\le r\le \mathsf k(U)-1\,.$$ It follows by the minimality of $\mathsf k(U)$ that $ UA_1\cdot\ldots A_{m-1}g_1^{\ord(g_1)}\cdot\ldots\cdot g_{m}^{\ord(g_m)}$ can be factorized into a product of atoms from $\Omega_{=1}$.
\qed{(Proof of {\bf A4})}

\bigskip
By {\bf A4}, we consider all  $I\in [1,m-1]$ and $J\in [1,\ell]$ such that $U\prod_{i\in I}A_i\prod_{j\in J}g_j^{\ord(g_j)}$ can be factorized into a product of atoms from $\Omega_{=1}$. We fix such $I$ and $J$ with $|I|+|J|$ is minimal. Then $|I|+|J|\le m-1+\ell\le r-1$. Since $J\neq \emptyset$, we choose $j_0\in J$ and hence $U\prod_{i\in I}A_i\prod_{j\in J\setminus\{j_0\}}g_j^{\ord(g_j)}$ can not be factorized into a product of atoms from $\Omega_{=1}$ by the minimality of $|I|+|J|$.

Now we consider all tuples $(U',A_1',\ldots, A_{m'-1}', g )$, where $U'\in \Omega_{> 1}'$, $m'\in\N$, $A_1',\ldots,A_{m'-1}'\in \Omega_{=1}$, and $g\in G_0$ such that $U'A_1'\cdot\ldots A_{m'-1}'g^{\ord(g)}$  can be factorized into a product of atoms from $\Omega_{=1}$ and $U'A_1'\cdot\ldots A_{m'-1}'$  can not be factorized into a product of atoms from $\Omega_{=1}$. We fix one such tuple $(U',A_1',\ldots, A_{m'-1}', g )$ with $m'$ is minimal. Thus $m'\le |I|+|J|\le r-1$.
Let $$U'A_1'\cdot\ldots A_{m'-1}'g^{\ord(g)}=W_1\cdot\ldots\cdot W_{t'}\,,\text{where }W_1,\ldots,W_{t'}\in \Omega_{=1}\,,$$ and we claim that

\begin{enumerate}
\item[{\bf  A5. }]  For each $\nu \in [1, t']$, we have $W_{\nu} \nmid U'A_1' \cdot \ldots \cdot A_{m'-1}'$.
\end{enumerate}

{\it Proof of {\bf A5.}} \, Assume to the contrary that there is such a $\nu \in [1, t']$, say $\nu = 1$, with  $W_1 \t U' A_1' \cdot \ldots \cdot A_{m'-1}'$. Then there are $l \in \N$ and $T_1, \ldots, T_l \in \mathcal A (G_0)$ such that
\[
U' A_1' \cdot \ldots \cdot A_{m'-1}' = W_1 T_1 \cdot \ldots \cdot T_{\mathit l} \,.
\]
Since $U'A_1'\cdot\ldots A_{m'-1}'$  can not be factorized into a product of atoms from $\Omega_{=1}$,  there exists some $\nu \in [1, l]$ such that $T_{\nu} \in \Omega_{>1}$, say $\nu=1$, and $T_1 \cdot \ldots \cdot T_{\mathit l}$ can not be factorized into a product of atoms from $\Omega_{=1}$.  Since
\[
\sum_{\nu=2}^l \mathsf k (T_{\nu}) = \mathsf k (U') + (m'-1)-1 - \mathsf k (T_1) \le m'-2 \le r-3\,,
\]
and $\mathsf k (T')\ge r+1$ for all $T' \in \Omega_{>1}$, it follows that  $T_2, \ldots, T_l \in \Omega_{=1}$, whence $l = 1 + \sum_{\nu=2}^l \mathsf k (T_{\nu}) \le m'-1$.
We obtain that
\[
W_1 T_1 \cdot \ldots \cdot T_{\mathit l}g^{\ord(g)} = U' A_1' \cdot \ldots \cdot A_{m'-1}g^{\ord(g)} = W_1 \cdot \ldots \cdot W_{t'} \,,
\] and thus
 \[
 T_1 \cdot \ldots \cdot T_{\mathit l}g^{\ord(g)} = W_2 \cdot \ldots \cdot W_{t'} \,.
\]
Since $T_1 \cdot \ldots \cdot T_{\mathit l}$ can not be factorized into a product of atoms from $\Omega_{=1}$, we obtain that  $\mathsf k(T_1)> \mathsf k(U)$ by  the minimality of $m'$.
It follows that  \[\mathsf k(T_1)-\mathsf k(U')=m'-1-{\mathit l}\le m'-2\le r-3<r\le \mathsf k(T_1)-\mathsf k(U)\,,\]
 a contradiction.  \qed{(Proof of {\bf  A5})}

\bigskip
Let $U'A_1'\cdot\ldots\cdot A_{m'-1}'=X_1'\cdot\ldots\cdot X_{t'}'$ and $g^{\ord(g)}=g^{y_1}\cdot\ldots\cdot g^{y_{t'}}$ such that $W_i=X_i'g^{y_i}$ for each $i\in [1,t']$. By {\bf A5}, we obtain that $y_i\ge 1$ for all $i\in [1,t']$. If $|\{i\in [1,t']\t y_i=1\}|\ge 2$, say $y_1=y_2=1$, then $\mathsf v_g(W_1)=\mathsf v_g(W_2)=1+\mathsf v_g(U'A_1'\cdot\ldots\cdot A_{m'-1}')$ by {\bf A5} and hence $\mathsf v_g(X_1X_2)=\mathsf v_g(W_1)+\mathsf v_g(W_2)-2=2\mathsf v_g(U'A_1'\cdot\ldots\cdot A_{m'-1}')\ge \mathsf v_g(U'A_1'\cdot\ldots\cdot A_{m'-1}')+\mathsf v_g(X_1X_2)$, a contradiction. Thus $|\{i\in [1,t']\t y_i=1\}|\le 1 $ and hence $1+2(t'-1)\le \ord(g)\le n$. It follows that
$$\mathsf k(U')=t'-m'\le \frac{n+1}{2}-1\le \left\lfloor\frac{n}{2}\right\rfloor\,,$$
a contradiction.
\end{proof}

\medskip
\begin{proposition} \label{3.7}
We have $\mathsf m(G) \le \max\{r-1, \left\lfloor\frac{n}{2}\right\rfloor-1 \}$.
\end{proposition}

\begin{proof}
Let $G_0 \subset G$ be a non-half-factorial LCN set. We have to prove that
\[
\min \Delta(G_0)\le \max\{r-1, \left\lfloor\frac{n}{2}\right\rfloor-1 \}\,.
\]
If $G_1 \subset G_0$ is non-half-factorial, then $\min \Delta (G_0) = \gcd \Delta (G_0) \t \gcd \Delta (G_1) = \min \Delta (G_1)$. Thus we may suppose that $G_0$ is minimal non-half-factorial. By Lemma \ref{3.1}.3.(a), we may suppose that $g \in \langle G_0 \setminus \{g\} \rangle$ for all $g \in G_0$.

If $|G_0| \le r+1$, then $\min \Delta (G_0) \le |G_0|-2\le r-1$ by Lemma \ref{3.2}.3. Thus we may suppose that  $|G_0| \ge r+2$ and we distinguish two cases.

\smallskip
\noindent
CASE 1: \, There exists a subset $G_2 \subset G_0$ such that $\langle G_2 \rangle = \langle G_0 \rangle$ and $|G_2| \le |G_0|-2$.

Then Lemma \ref{3.6} implies that $\min \Delta (G_0)\le \max\{r-1, \left\lfloor\frac{n}{2}\right\rfloor-1 \}$.

\smallskip
\noindent
CASE 2: \, Every subset $G_1\subset G_0$ with $|G_1|=|G_0|-1$ is a minimal generating set of $\langle G_0 \rangle$.

Then for each $h \in G_0$, $G_0\setminus \{h\}$ is half-factorial and $h \notin \langle G_0 \setminus \{h, h' \} \rangle$ for any $ h'\in G_0\setminus\{h\}$. It follows that Lemma \ref{3.4} and Lemma \ref{3.6} imply that $\min \Delta(G_0)\le \max\{r-1, \left\lfloor\frac{n}{2}\right\rfloor-1 \}$.
\end{proof}

\medskip
\section{Proofs of the main theorems}\label{4}
\medskip

In this section we give the proofs of Theorems \ref{1.1} and \ref{1.2}.

\medskip
\begin{proof}[Proof of Theorem \ref{1.1}]
Let $H$ be a Krull monoid with finite class group $G$ where $|G| \ge 3$ and   every class contains a prime divisor.
We set $\exp (G)=n$, $\mathsf r (G)=r$, and let $k\in \N$ be maximal such that $G$ has a subgroup isomorphic to $C_{n}^k$.
By Lemma \ref{2.1}, it suffices to prove the assertions for the Krull monoid $\mathcal B (G)$.

Propositions \ref{2.3}.3  and \ref{3.7} immediately imply the required inclusions for $\Delta^* (G)$, namely that
 \begin{equation}\label{4.1}
 \begin{aligned}&[1,r-1]\cup\{\max\{1,\lfloor\frac{n}{2}\rfloor-1\}\}\cup [\max\{1,n-k-1\},n-2]\\
 \subset \Delta^*(G)\subset &[1,\max\{r-1, \lfloor\frac{n}{2}\rfloor-1 \}]\cup [\max\{1,n-k-1\},n-2]\,.
 \end{aligned}
 \end{equation}
It remains to verify the in particular statements.

\smallskip
1. If  $r\ge \left\lfloor\frac{n}{2}\right\rfloor-1$, then  $[1,\max\{r-1, \lfloor\frac{n}{2}\rfloor-1 \}]\subset [1,r-1]\cup\{\max\{1,\lfloor\frac{n}{2}\rfloor-1\}\}$. Therefore  $ \Delta^*(G)=[1,\max\{r-1, \lfloor\frac{n}{2}\rfloor-1 \}]\cup [\max\{1,n-k-1\},n-2]\,$ by Equation \eqref{4.1}.

\smallskip
2. (a) $\Rightarrow$ (b) \ Suppose that $\Delta^*(G)$ is an interval. Since $\max\{1,n-k-2\}\le \max\{r-1,n-2\}=\max\Delta^*(G)$, we obtain that $\max\{1,n-k-2\}\in \Delta^*(G)$.

   (b) $\Rightarrow$ (c) \
    Suppose that $\max\{1,n-k-2\}\in \Delta^*(G)$. If $n-k-2\le 0$, then  $n-k-2\le \max\{r-1,\lfloor\frac{n}{2}\rfloor-1\}$. If $n-k-2\ge 1$, then $n-k-2\in \Delta^*(G)\subset [1,\max\{r-1, \lfloor\frac{n}{2}\rfloor-1 \}]\cup [n-k-1,n-2]$ by Equation \ref{4.1}. Therefore $n-k-2\le \max\{r-1, \lfloor\frac{n}{2}\rfloor-1\} $.

    (c) $\Rightarrow$ (d) \
     Suppose that $n-k-2\le \max\{r-1, \lfloor\frac{n}{2}\rfloor-1 \}$. Therefore $n-k-2\le r-1$ or $r\le n-k-2\le \lfloor\frac{n}{2}\rfloor-1$. If $n-k-2\le r-1$, then $r+k\ge n-1$. If $r\le n-k-2\le \lfloor\frac{n}{2}\rfloor-1$, then $n-r-2\le n-k-2\le \lfloor\frac{n}{2}\rfloor-1\le \frac{n}{2}-1$ and $r\le  \lfloor\frac{n}{2}\rfloor-1\le \frac{n}{2}-1$. It follows that $n-2=n-r-2+r\le \frac{n}{2}-1+\frac{n}{2}-1=n-2$ which implies that
   $n-r-2=n-k-2=\frac{n}{2}-1$ and $r=\frac{n}{2}-1$. Therefore $r=k$, $n=2r+2$, and hence $G\cong C_{2r+2}^r$.

   (d) $\Rightarrow$ (a) \
   If $r+k=n-2$ and $G\cong C_{2r+2}^{r}$, then $\Delta^*(G)=[1,2r]$ is an interval by 1.
   If  $r+k\ge n-1$,  then $r\ge \left\lfloor\frac{n}{2}\right\rfloor$ and hence $\Delta^*(G)=[1,r-1]\cup [\max\{1,n-k-1\},n-2]$ is an interval by 1.
   \end{proof}

\medskip

\begin{proof}[Proof of Theorem \ref{1.2}]
 Let $G$ and $G'$ be finite abelian groups with $\exp(G)=n$ and  $\mathsf r (G)=r$. Let $k, k' \in \N$ be maximal such that $G$ has a subgroup isomorphic to $C_n^k$ and $G'$ has a subgroup isomorphic to $C_{\exp(G')}^{k'}$. Suppose that
\[
r+k \le n-2, \quad  G \not\cong C_{2r+2}^r, \quad \text{ and  that} \quad  \mathcal L (G) = \mathcal L (G') \,.
\]

By our assumption and Theorem \ref{1.1}.2, we have that $\Delta^*(G)$ is not an interval, $n-k-2\not\in \Delta^*(G)$, and  $n-k-2\ge \max\{r,\left\lfloor\frac{n}{2}\right\rfloor\} $. By Proposition \ref{2.3}, we obtain that $\max\Delta_1(G)=\max\Delta^*(G)=\max\{r-1,n-2\}=n-2$, $n-k-2\not\in \Delta_1(G)$, and $n-k-1\in \Delta_1(G)$. Note that $\mathsf D(G)=\mathsf D(G')$ and $\Delta_1(G)=\Delta_1(G')$ (see \cite[Proposition 7.3.1]{Ge-HK06a}).
Then $\max \Delta_1(G')=\max\{\mathsf r(G')-1, \exp(G')-2\}=\max\Delta_1(G)=n-2$, $n-k-2\not\in \Delta_1(G')$, $n-k-1\in \Delta_1(G')$. If $\mathsf r(G')\ge \exp(G')-1$, then $\Delta_1(G')=[1,\mathsf r(G')-1]$ by Proposition \ref{2.3}, a contradiction.  It follows that $\exp(G')=n$ by $\max\Delta_1(G')=\exp(G')-2$. Suppose that $k'\ge k+1$. Then $n-k-2\in [n-k'-1, n-2]\subset \Delta_1(G')=\Delta_1(G)$, a contradiction. Suppose that $k'\le k-1$. Then $n-k-1\not\in [n-k'-1, n-2]$ and hence $n-k-1\in [1,\max\{\mathsf r(G')-1, \left\lfloor\frac{n}{2}\right\rfloor-1\}]$. If $n-k-1\le \mathsf r(G')-1$, then $n-k-2\in [1,\mathsf r(G')-1]\subset \Delta_1(G')=\Delta_1(G)$, a contradiction. Otherwise $ n-k-1\le \left\lfloor\frac{n}{2}\right\rfloor-1$, a contradiction to $n-k-2\ge \left\lfloor\frac{n}{2}\right\rfloor $. It follows that $k=k'$.

In particular, if $ r\ge \left\lfloor\frac{n}{2}\right\rfloor+1$, then $[1,r-1]\cup [n-k-1,n-2]=\Delta_1(G)=\Delta_1(G')$ and hence
$[1,\mathsf r(G')]\subset[1,r-1]\subset [1,\max\{\mathsf r(G')-1,\left\lfloor\frac{n}{2}\right\rfloor-1\}]$. Therefore by $ r\ge \left\lfloor\frac{n}{2}\right\rfloor+1$ we obtain that $\mathsf r(G')=r$.

If $\mathsf r(G)=k$, then $G\cong C_n^r$ is a subgroup of $G'$. Thus $\mathsf D(G)=\mathsf D(G')$ implies that $G\cong G'$.
\end{proof}

\providecommand{\bysame}{\leavevmode\hbox to3em{\hrulefill}\thinspace}
\providecommand{\MR}{\relax\ifhmode\unskip\space\fi MR }
\providecommand{\MRhref}[2]{%
  \href{http://www.ams.org/mathscinet-getitem?mr=#1}{#2}
}
\providecommand{\href}[2]{#2}


\begin{thebibliography}{10}

\bibitem{Ba-Ge14b}
N.R. Baeth and A.~Geroldinger, \emph{Monoids of modules and arithmetic of
  direct-sum decompositions}, Pacific J. Math. \textbf{271} (2014), 257 -- 319.

\bibitem{Ch11a}
Gyu~Whan Chang, \emph{Every divisor class of {K}rull monoid domains contains a
  prime ideal}, J. Algebra \textbf{336} (2011), 370 -- 377.

\bibitem{C-F-G-O16}
S.T. Chapman, M.~Fontana, A.~Geroldinger, and B.~Olberding (eds.),
  \emph{Multiplicative {I}deal {T}heory and {F}actorization {T}heory}, vol.
  170, Springer,
  \href{http://www.springer.com/de/book/9783319388533?token=prtst0416p}{Proceedings
  in Mathematics and Statistics}, 2016.

\bibitem{Ch-Sc-Sm08b}
S.T. Chapman, W.A. Schmid, and W.W. Smith, \emph{On minimal distances in
  {K}rull monoids with infinite class group}, Bull. Lond. Math. Soc.
  \textbf{40} (2008), 613 -- 618.

\bibitem{Fa06a}
A.~Facchini, \emph{Krull monoids and their application in module theory},
  Algebras, {R}ings and their {R}epresentations (A.~Facchini, K.~Fuller, C.~M.
  Ringel, and C.~Santa-Clara, eds.), World Scientific, 2006, pp.~53 -- 71.

\bibitem{Ge16c}
A.~Geroldinger, \emph{Sets of lengths},
  \href{http://arxiv.org/abs/1509.07462}{arXiv:1509.07462}.

\bibitem{Ge-Gr-Sc11a}
A.~Geroldinger, D.J. Grynkiewicz, and W.A. Schmid, \emph{The catenary degree of
  {K}rull monoids {I}}, J. Th{\'e}or. Nombres Bordx. \textbf{23} (2011), 137 --
  169.

\bibitem{Ge-HK06a}
A.~Geroldinger and F.~Halter-Koch, \emph{Non-{U}nique {F}actorizations.
  {A}lgebraic, {C}ombinatorial and {A}nalytic {T}heory}, Pure and Applied
  Mathematics, vol. 278, Chapman \& Hall/CRC, 2006.

\bibitem{Ge-Ha02}
A.~Geroldinger and {Y. O.} Hamidoune, \emph{Zero-sumfree sequences in cyclic
  groups and some arithmetical application}, J. Th{\'e}or. Nombres Bordx.
  \textbf{14} (2002), 221 -- 239.

\bibitem{Ge-Ru09}
A.~Geroldinger and I.~Ruzsa, \emph{Combinatorial {N}umber {T}heory and
  {A}dditive {G}roup {T}heory}, Advanced Courses in Mathematics - CRM
  Barcelona, Birkh{\"a}user, 2009.

\bibitem{Ge-Sc16a}
A.~Geroldinger and W.~A. Schmid, \emph{A characterization of class groups via
  sets of lengths}, \href{http://arxiv.org/abs/1503.04679}{arXiv:1503.04679}.

\bibitem{Ge-Sc16b}
\bysame, \emph{The system of sets of lengths in {K}rull monoids under set
  addition}, Rev. Mat. Iberoam. \textbf{32} (2016), 571 -- 588.

\bibitem{Ge-Yu12b}
A.~Geroldinger and P.~Yuan, \emph{The set of distances in {K}rull monoids},
  Bull. Lond. Math. Soc. \textbf{44} (2012), 1203 –-- 1208.

\bibitem{Ge-Zh16b}
A.~Geroldinger and Q.~Zhong, \emph{A characterization of class groups via sets
  of lengths {II}}, J. Th{\'e}or. Nombres Bordx., to appear.

\bibitem{Ge-Zh15b}
\bysame, \emph{The catenary degree of {K}rull monoids {II}}, J. Australian
  Math. Soc. \textbf{98} (2015), 324 -- 354.

\bibitem{Ge-Zh16a}
\bysame, \emph{The set of minimal distances in {K}rull monoids}, Acta Arith.
  \textbf{173} (2016), 97 -- 120.

\bibitem{Gr13a}
D.J. Grynkiewicz, \emph{Structural {A}dditive {T}heory}, Developments in
  Mathematics, Springer, 2013.

\bibitem{HK98}
F.~Halter-Koch, \emph{Ideal {S}ystems. {A}n {I}ntroduction to {M}ultiplicative
  {I}deal {T}heory}, Marcel Dekker, 1998.

\bibitem{Ki-Pa01}
H.~Kim and Y.~S. Park, \emph{{K}rull domains of generalized power series}, J.
  Algebra \textbf{237} (2001), 292 -- 301.

\bibitem{Pl-Sc16a}
A.~Plagne and W.A. Schmid, \emph{On congruence half-factorial {K}rull monoids
  with cyclic class group}, submitted.

\bibitem{Sc05d}
W.A. Schmid, \emph{Differences in sets of lengths of {K}rull monoids with
  finite class group}, J. Th{\'e}or. Nombres Bordx. \textbf{17} (2005), 323 --
  345.

\bibitem{Sc09c}
\bysame, \emph{Arithmetical characterization of class groups of the form
  $\mathbb{Z} /n \mathbb{Z} \oplus \mathbb{Z} /n \mathbb{Z}$ via the system of
  sets of lengths}, Abh. Math. Semin. Univ. Hamb. \textbf{79} (2009), 25 -- 35.

\bibitem{Sc09b}
\bysame, \emph{Characterization of class groups of {K}rull monoids via their
  systems of sets of lengths{\rm \,:} a status report}, Number {T}heory and
  {A}pplications{\rm \,:} {P}roceedings of the {I}nternational {C}onferences on
  {N}umber {T}heory and {C}ryptography (S.D. Adhikari and B.~Ramakrishnan,
  eds.), Hindustan Book Agency, 2009, pp.~189 -- 212.

\bibitem{Sc09a}
\bysame, \emph{A realization theorem for sets of lengths}, J. Number Theory
  \textbf{129} (2009), 990 -- 999.

\bibitem{Sm13a}
D.~Smertnig, \emph{Sets of lengths in maximal orders in central simple
  algebras}, J. Algebra \textbf{390} (2013), 1 -- 43.

\end{thebibliography}
\end{document}